\documentclass[12pt]{amsart}
\usepackage{amsmath}
\usepackage{amsfonts}
\usepackage{amssymb}
\usepackage{amsthm}
\usepackage[all]{xy}
\usepackage{color}
\usepackage{verbatim}
\usepackage{graphicx}

\newtheorem{theorem}{Theorem}[section]
\newtheorem{lemma}[theorem]{Lemma}
\newtheorem{proposition}[theorem]{Proposition}
\newtheorem{corollary}[theorem]{Corollary}
\newtheorem{definition}[theorem]{Definition}

\newcommand{\RN}[1]{%
  \textup{\uppercase\expandafter{\romannumeral#1}}%
}

\numberwithin{equation}{section}

 \author[Nupur Patanker]{Nupur Patanker}
 
  \address{ Indian Institute of Science Education and Research, Bhopal}
 
 \email{nupurp@iiserb.ac.in}

  \author[Sanjay Kumar Singh]{Sanjay Kumar Singh}
    \address{Indian Institute of Science Education and Research, Bhopal}
  \email{sanjayks@iiserb.ac.in}

\title[generalized Hamming weights of codes] {Generalized Hamming weights of toric codes over hypersimplices and square-free affine evaluation codes}

\date{}
\begin{document}
\begin{abstract}
Let $\mathbb{F}_{q}$ be a finite field with $q$ elements, where $q$ is a power of prime $p$. A  polynomial over $\mathbb{F}_{q}$ is square-free if all its monomials are square-free. In this note, we determine an upper bound on the number of zeroes in the affine torus $T=(\mathbb{F}_{q}^{*})^{s}$ of any set of $r$ linearly independent square-free polynomials over $\mathbb{F}_{q}$ in $s$ variables, under certain conditions on $r$, $s$ and degree of these polynomials. Applying the results, we partly obtain the generalized Hamming weights of toric codes over hypersimplices and square-free evaluation codes, as defined in \cite{hyper}. Finally, we obtain the dual of these toric codes with respect to the Euclidean scalar product.


\end{abstract}
\maketitle
\section{Introduction} 
The fundamental parameters of linear codes, such as dimension and minimum distance, determine the efficiency and error-correction capability of the codes. Another important property of linear codes is their generalized Hamming weights. The notion of generalized Hamming weights for a linear code $C$ over $\mathbb{F}_{q}$ is defined as follows.\\

 For any $\mathbb{F}_{q}$-subspace $D$ of $[n,k]$ code $C$, the support of $D$ is defined as
$$supp(D):=\{ 1 \leq i \leq n~:~ x_{i} \neq 0 \text{ for some } \mathbf{x}=(x_{1},\cdots,x_{n}) \in D\}.$$
For $1 \leq r \leq k$, the \textit{$r$-th generalized Hamming weight} of $C$ is defined as 
$$d_{r}(C):=min~\{~ |supp(D)|~:~ D \text{ is a linear subcode of } C \text { with } dim(D)=r\}.$$\\
In particular, the first generalized Hamming weight of $C$ is the usual minimum distance. The set of generalized Hamming weights $\{d_{1}(C),~d_{2}(C), \cdots, d_{k}(C)\}$ is called the \textit{weight hierarchy} of code $C$. The notions of generalized Hamming weights for linear codes were introduced in \cite{ghw_in}, \cite{ghw2}, and rediscovered by Wei in his paper \cite{ghw}. These weights completely characterize the performance of the code on the wire-tap channel of type $\RN{2}$, and also the performance as a $t$-resilient function. The generalized Hamming weights of various linear codes have been studied for many years.\par

Toric codes were introduced by J. Hansen in \cite{toric} and since then have been studied in \cite{tor1}, \cite{tor2}, \cite{tor3}, \cite{tor4}, \cite{tor5}, \cite{tor6}, \cite{tor7}, \cite{tor8}, etc. Projective Reed-Muller-type code over the projective torus has been studied in \cite{tor9}, \cite{tor10}, etc. Recently, Delio Jaramillo, Maria Vaz Pinto and Rafael H. Villarreal, in \cite{hyper}, introduced affine and projective toric codes over hypersimplices. The authors computed their dimension and minimum distance. They also introduced square-free evaluation codes and computed their dimension, minimum distance and second generalized Hamming weight. They posed the problem of obtaining formulae for the generalized Hamming weights of these codes. In this note, we determine the generalized Hamming weights of toric codes over hypersimplices and square-free affine evaluation code.\par

The problem of finding the generalized Hamming weights of toric codes over hypersimplices can be solved by answering the following question stated in terms of polynomials:\par
Let $s$ and $d$ be integers such that $s\geq2$ and $1 \leq d \leq s$. For $1 \leq r \leq  { s \choose d}$, let $f_{1},~f_{2},\cdots,f_{r}$ be linearly independent homogeneous square-free polynomials of degree $d$ in $s$ variables with coefficients in $\mathbb{F}_{q}$. What is the maximum number of solutions in affine torus $T=(\mathbb{F}_{q}^{*})^s$ of the system $f_{1}=f_{2}=\cdots=f_{r}=0$?\par
In \cite{hyper}, the answer to this problem is given for $r=1$. Our goal in this note is to solve a more generalized problem where $f_{1},~f_{2},\cdots,f_{r}$ are linearly independent square-free polynomials of degree $d$ in $s$ variables with coefficients in $\mathbb{F}_{q}$. To obtain our results, we follow the footsteps of \cite{acc}. Another related question is to solve the above-stated problem when $f_{1},~f_{2},\cdots,f_{r}$ are linearly independent square-free polynomials of degree at most $d$ in $s$ variables with coefficients in $\mathbb{F}_{q}$, where $1 \leq r \leq \sum_{i=0}^{d} {s \choose i}$. The answer to this problem helps us to determine the generalized Hamming weights of square-free affine evaluation codes. The answer for $r=1,2$ is already given in \cite{hyper}. In this note, we answer these questions when $d+r-2<s$ and as an application, determine the generalized Hamming weights of these codes.\par 

This note is organized as follows. In section $2$, we recall the definition of toric code over hypersimplices and square-free evaluation codes, as defined in \cite{hyper}. We also study the affine Hilbert function. In section $3$, we determine an upper bound on the number of solutions in the affine torus of any set of $r$ linearly independent square-free polynomials over $\mathbb{F}_{q}$ of degree $d$ in $s$ variables, $1 \leq r \leq  { s \choose d}$. We also determine an upper bound on the number of solutions in affine torus of any set of $r$ linearly independent square-free polynomials over $\mathbb{F}_{q}$ of degree at most $d$ in $s$ variables, $1 \leq r \leq \sum_{i=0}^{d} {s \choose i}$. In section $4$, we determine the generalized Hamming weights of the toric codes over hypersimplices and square-free evaluation codes in specific cases. In section $5$, we conclude the note by determining the dual of toric codes over hypersimplices with respect to the Euclidean scalar product.
\section{Preliminaries}
Let $s$ and $d$ be integers such that $s \geq 2$ and $ 1 \leq d \leq s$. In this section, we recall the definitions of toric codes over hypersimplices and square-free affine evaluation codes. We also recall the known results on these codes and study the affine Hilbert function.\par
Throughout this note, we use the notation $K:=\mathbb{F}_{q}$, where $q$ is a power of prime $p$.
\subsection{Evaluation codes over $d$-th hypersimplex, \cite{hyper} }
Let $S:=K[t_{1},\cdots,t_{s}]=\bigoplus_{d=0}^{\infty} S_{d}$ be the polynomial ring in $s$ variables over $K$ with standard grading.\par Let $\mathcal{P}$ be the convex hull in $\mathbb{R}^{s}$ of all integral points $e_{i_1}+e_{i_2}+\cdots+e_{i_d}$ such that $1 \leq i_{1} <\cdots < i_{d} \leq s$, where $e_{i}$ is the $i$-th unit vector in $\mathbb{R}^{s}$. The lattice polytope $\mathcal{P}$ is called the $d$-th hypersimplex in $\mathbb{R}^{s}$. The affine torus of the affine space $\mathbb{A}^{s}$ is given by $T := (K^{*})^{s}$, where $K^{*}$ is the multiplicative group of $K$. The projective torus of the projective space $\mathbb{P}^{s-1}$ over $K$ is given by $\mathbb{T} := [T]$, where $[T]$ is the image of $T$ under the map $\phi:\mathbb{A}^{s} \backslash \{0\} \rightarrow \mathbb{P}^{s-1}$, $a \mapsto [a]$. The cardinality of $T$ is $m:=(q-1)^{s}$ and the cardinality of $\mathbb{T}$ is $\bar{m}:=(q-1)^{s-1}$. Let $V_{d}$ be the set all monomials $t^{a}:=t_{1}^{a_{1}} t_{2}^{a_{2}} \cdots t_{s}^{a_{s}}$ such that $a \in \mathcal{P} \cap \mathbb{Z}^{s}$ and let $KV_{d}$ be the vector space over $K$ generated by $V_{d}$. Thus, $KV_{d}$ is the space of homogeneous square-free polynomials of $S$ of degree $d$. Denote by $P_{1},~P_{2},\cdots, P_{m}$ all points of the affine torus $T$ of $\mathbb{A}^{s}$ and denote by $[Q_{1}],~[Q_{2}],\cdots, [Q_{\bar{m}}]$ all points of the projective torus $\mathbb{T}$ of $\mathbb{P}^{s-1}$. We assume that the first entry of each $Q_{i}$ is $1$. Thus, $\mathbb{T}=\{1\} \times (\mathbb{F}_{q}^{*})^{s-1}$.\par 
 The affine toric code $C_d$ of $\mathcal{P}$ of degree $d$ is defined as the image of the evaluation map
\begin{equation}
ev_{d}:KV_{d} \rightarrow \mathbb{F}_{q}^{m},~~  ev_{d}(f):=(f(P_{1}),~f(P_{2}), \cdots, f(P_{m})).
\end{equation}
The code $C_d$ has length $m$. The minimum distance of $C_d$ is given by $$\delta({C_d}):=min \{~ |T \backslash V_{T}(f)| : ~f \in KV_{d} \backslash I(T)\},$$
where $V_{T}(f)$ denotes the set of zeroes of $f \in S$ in $T$.\\

The projective toric code $C^{\mathbb{P}}_{d}$ of $\mathcal{P}$ of degree $d$ is defined as the image of the evaluation map
\begin{equation}
ev_{d}:KV_{d} \rightarrow \mathbb{F}_{q}^{\bar{m}},~~  ev_{d}(g):=(g(Q_{1}),~g(Q_{2}), \cdots, g(Q_{\bar{m}})).
\end{equation}
The code $C^{\mathbb{P}}_{d}$ has length $\bar{m}$. The minimum distance of $C^{\mathbb{P}}_{d}$ is given by $$\delta({C^{\mathbb{P}}_{d}}):=min\{ ~| \mathbb{T} \backslash V_{\mathbb{T}}(g)| : ~g \in KV_{d} \backslash I(\mathbb{T})\},$$
where $V_{\mathbb{T}}(g)$ denotes the set of zeroes of $g \in S$ in $\mathbb{T}$.\\
 
The dimension and minimum distance of $C_{d}$ and $C^{\mathbb{P}}_{d}$ are given by the following theorems.

\begin{theorem}$($\cite{hyper}, Proposition $4.4)$
Let $C_{d}$ and $C^{\mathbb{P}}_{d}$ be the affine and projective toric code of $\mathcal{P}$ of degree $d$, respectively. Then
\[
dim_{K}(C_{d})= dim_{K}(C^{\mathbb{P}}_{d}) =\begin{cases}
 {s \choose d}, & \text{if } q \geq 3,\\
 1, &\text{if } q=2.\\
 \end{cases}  
 \]
\end{theorem}
~\\
\begin{theorem} $($\cite{hyper}, Theorem $4.5)$
Let $C_{d}$ be the affine toric code of $\mathcal{P}$ of degree $d$ and let $\delta(C_{d})$ be its minimum distance. Then
\[
 \delta(C_{d}) =\begin{cases}
(q-2)^{d}(q-1)^{s-d}, & \text{if } d \leq s/2, ~q \geq 3,\\
 (q-2)^{s-d}(q-1)^{d}, &\text{if } s/2<d<s,~ q\geq 3,\\
 (q-1)^{s}, &\text{if } d=s,\\
 1, &\text{if } q=2.\\
 \end{cases}  
 \]
and let $C^{\mathbb{P}}_{d}$ be the projective toric code of $\mathcal{P}$ of degree $d$ and let $\delta(C^{\mathbb{P}}_{d})$ be its minimum distance. Then
\[
 \delta(C^{\mathbb{P}}_{d}) =\begin{cases}
(q-2)^{d}(q-1)^{s-d-1}, & \text{if } d \leq s/2, ~q \geq 3,\\
 (q-2)^{s-d}(q-1)^{d-1}, &\text{if } s/2<d<s,~ q\geq 3,\\
 (q-1)^{s-1}, &\text{if } d=s,\\
 1, &\text{if } q=2.\\
 \end{cases}  
 \] 
\end{theorem}

\subsection{Square-free affine evaluation code}
Let $V_{\leq d}$ be the set of all square-free monomials of $S$ of degree at most $d$ and $KV_{\leq d}$ be the corresponding subspace of $S_{\leq d}$. If we replace $KV_{d}$ by $KV_{\leq d}$ in the evaluation map of equation $(2.1)$, the image
of the resulting map, denoted $C_{\leq d}$, is called a square-free affine evaluation code of degree $d$ on $T$.

The following results, proved in \cite{hyper} give the dimension, minimum distance and second generalized Hamming weight of $C_{\leq d}$.

\begin{proposition}{$($\cite{hyper}, Proposition $5.2)$}
Let $C_{\leq d}$ be the square-free affine evaluation code of degree $d$ on the affine torus $T = (K^{*})^{s}$. Then, the length of $C_{\leq d}$ is $(q-1)^{s}$, and the dimension of $C_{\leq d}$ is given by
\[
dim_{K}(C_{\leq d}) =\begin{cases}
 {s \choose 0}+ {s \choose 1}+ \cdots+{s \choose d}, & \text{if } q \geq 3,\\
 1, &\text{if } q=2.\\
 \end{cases}  
 \]
\end{proposition}

\begin{theorem}{$($\cite{hyper}, Theorem $5.5)$}
If $q \geq 3$, then the minimum distance $\delta(C_{\leq d})$ of $C_{\leq d}$ is $(q-2)^{d} (q-1)^{s-d}$.\\
\end{theorem}

\begin{theorem}{ $($\cite{hyper}, Theorem $5.6)$}
If $q\geq 3$ and $d \geq 1$, then the second generalized Hamming weight of $C_{\leq d}$ is
\[
\delta_{2}(C_{\leq d}) =\begin{cases}
(q-2)^{s-1}(q-1), & \text{if } d=s,\\
 (q-2)^{d}(q-1)^{s-d-1}q, &\text{if } d<s.\\
 \end{cases}  
 \]
\end{theorem}

\subsection{Affine Hilbert function} In this subsection, we briefly discuss the affine Hilbert function of an ideal $I \subset K[t_{1},\cdots,t_{s}] $. For more details on this topic refer to \cite{book} and \cite{ahf}.\par
Let $K[t_{1},\cdots,t_{s}]_{\leq u}$ denotes the subset of $K[t_{1},\cdots,t_{s}]$ consisting of polynomials of total degree $\leq u$. For an ideal $I \subset K[t_{1},\cdots,t_{s}]$, we denote by $I_{\leq u}$ the subset of $I$ consisting of polynomials of degree $\leq u$.
\begin{definition}
The \textit{affine Hilbert function} of $I$ is the function on the non-negative integers $u$ defined by $$^{a} \mathrm{{HF}_{I}} (u) := dim_{K}~K[t_{1},\cdots,t_{s}]_{\leq u}/I_{\leq u}= dim_{K}~K[t_{1},\cdots,t_{s}]_{\leq u} -dim_{K}~ I_{\leq u}.$$
\end{definition}

Note that if $I\subseteq J$ are any ideals of $K[t_{1},\cdots,t_{s}]$, then $^{a} \mathrm{{HF}_{I}}(u) \geq$ $^{a} \mathrm{{HF}_{J}}(u)$. Given a subset $X$ of $K^{s}$, let $I(X)$ denotes the vanishing ideal of $X$ in $K[t_{1},\cdots,t_{s}]$. Then the affine Hilbert function of $X$, denoted by $^{a} \mathrm{{HF}_{X}}(u)$, is defined as $^{a} \mathrm{{HF}_{X}}(u):=$ $^{a} \mathrm{{HF}_{I(X)}}(u)$. \\

We have the following result on the affine Hilbert function of an ideal of $K[t_{1},\cdots,t_{s}]$. The proof can be found in \cite{book}, Chapter $9$, section $3$.
\begin{proposition}
Fix a graded monomial ordering $\prec$ on $K[t_{1},\cdots,t_{s}]$, then 
\begin{enumerate}
\item For any ideal $I$ of $K[t_{1},\cdots,t_{s}]$, we have $^{a}\mathrm{{HF}_{I}}(u)=$ $^{a}\mathrm{{HF}_{\langle LT(I) \rangle}}(u)$.
\item If $I$ is a monomial ideal of $K[t_{1},\cdots,t_{s}]$, then $^{a}\mathrm{{HF}_{I}}(u)$ is the number of monomials of degree at most $u$ that does not lie in $I$.
\end{enumerate}
\end{proposition}

Another important result is the following proposition which can be found in \cite{ahf}, Lemma $2.1$. A similar statement can be found in \cite{aff}, Corollary $4.5$ and \cite{acc}.
\begin{proposition} {$($\cite{acc}, Proposition $2.2)$}
Let $Y \subseteq K^s$ be a finite set. Then, $|Y|=$ $^{a}\mathrm{{HF}_{Y}}(u)$ for sufficiently large $u$.
\end{proposition}

\section{Zeroes of square-free polynomials in the affine torus $T=(\mathbb{F}_{q}^{*})^{s} \subseteq \mathbb{A}^{s}$}
 Throughout this section, we take $\prec$  to be the standard graded lexicographic order on $S$ with $t_{s} \prec \cdots \prec t_{2} \prec  t_{1}$.\\
 
For two distinct square-free polynomials $f$ and $g$ in $S$ of degree $d$ in $s$ variables, the following two lemmas give an upper bound on the cardinality of the sets $V_{T}(f)$ and $V_{T}(f) \cap V_{T}(g)$. Lemma $3.2$ has been proved in \cite{hyper}, Proposition $4.3$.  We give another proof of the proposition. First, we need the following lemma from \cite{hyper}. We add the proof for the convenience of the reader. 
 
\begin{lemma}
Let $h$ be a square-free polynomial in $S \backslash \mathbb{F}_{q}$. If $h = (t_{1}-\alpha)h_{1}$ for some $\alpha \in \mathbb{F}_{q}^{*}$ and $h_{1} \in S$, then $h_{1}$ is a square-free polynomial in the variables $t_{2},~t_{3}, \cdots, t_{s}$.
\end{lemma}
\begin{proof}
Let $h_{1}=\sum_{i=1}^{w}\beta_{i}f_{i}$ where $\beta_{i} \in \mathbb{F}_{q}^{*}$, $1 \leq i \leq w$ and $f_{1},~f_{2},\cdots,f_{w}$ are distinct monomials. Then
\begin{equation}
h= \beta_{1}t_{1}f_{1}+\cdots + \beta_{w}t_{1}f_{w}- \alpha \beta_{1}f_{1}-\cdots- \alpha \beta_{w}f_{w}.
\end{equation} 
Assume that $t_{1}$ divides $f_{j}$ for some $1 \leq j \leq w$ and choose $j$ and $n\geq 1 $ such that $t_{1}^{n}$ divides $f_{j}$ and $t_{1}^{n+1}$ does not divides $f_{i}$ for $i=1,\cdots, w$. As $h$ is square-free, by equation $(3.1)$, the monomial $t_{1}f_{j}$ must be equal to $f_{l}$ for some $1\leq l \leq w$, a contradiction because $t_{1}^{n+1}$ does not divides $f_{l}$. This shows that $h_{1}$ is a polynomial in the variables $t_{2},~t_{3},\cdots,t_{s}$. Hence $t_{1}f_{1},~t_{1}f_{2},\cdots, t_{1}f_{w}, f_{1},~f_{2},\cdots, f_{w}$ are distinct monomials. As $h$ is square-free, by equation $(3.1)$, $f_{i}$ is square-free for $i = 1,~2,\cdots, w$, i.e.\  $h_{1}$ is square-free.\\
\end{proof} 
 
\begin{lemma}
Let $s \geq 2$ and $1 \leq d \leq s$. For any non-zero square-free polynomial $g$ of degree $d$ in $\mathbb{F}_{q}[t_{1},t_{2},\cdots,t_{s}]$, we have $$| V_{T}(g)| \leq (q-1)^{s}-(q-2)^{d}(q-1)^{s-d}.$$
\end{lemma} 
\begin{proof}
We prove this lemma by induction on $s$. For $s=2$, we have either $d=1$ or $d=2$. \\
When $d=1$, we have to show that $|V_{T}(g)|\leq (q-1)$. By direct calculations, we obtain the following table, where $\lambda, \mu, \delta \in \mathbb{F}_{q}^{*}$,\\
 \begin{center}
\begin{tabular}{ |l|l| } 
 \hline
  $g$ & $| V_{T}(g)|$\\ 
 \hline
  $\lambda t_{1}+ \mu $ & $q-1$  \\ 
  $\lambda t_{1}$ &$0$\\
  $\lambda t_{2}$ &$0$\\
  $\lambda t_{2}+ \mu $ & $q-1$\\
  $\lambda t_{1}+ \mu t_{2}$ & $q-1$\\
  $\lambda t_{1}+ \mu t_{2}+\delta $ & $q-2$\\
 \hline
\end{tabular}
\end{center}
The first column contains the various choices of polynomial $g$ in two variables $t_{1}$, $t_{2}$ of degree one and the second column specifies the number of zeroes in $T$ of the corresponding polynomial. From the above table, we have $|V_{T}(g)|\leq (q-1)$.\\\\
When $d=2$, we have to show that $|V_{T}(g)| \leq 2q-3.$ By direct calculations, we have the following table, where $\lambda, \mu, \delta, \rho \in \mathbb{F}_{q}^{*}$,\\
 \begin{center}
\begin{tabular}{ |l|l| } 
 \hline
  $g$ & $|V_{T}(g)|$\\ 
 \hline
 $\lambda t_{1}t_{2}$ & $0$  \\ 
  $\lambda t_{1}t_{2}+ \mu $ & $q-1$  \\ 
  $\lambda t_{1}t_{2}+\mu t_{1}$ &$q-1$\\
  $\lambda t_{1}t_{2}+\mu t_{1}+\delta$ &$q-2$\\
  $\lambda t_{1}t_{2}+ \mu t_{2} $ & $q-1$\\
  $\lambda t_{1}t_{2}+ \mu t_{2}+\delta$ & $q-2$\\
  $\lambda t_{1}t_{2}+ \mu t_{1}+\delta t_{2}+\rho $ & $\leq 2q-3$\\
 \hline
\end{tabular}
\end{center}
From the above table, we have $|V_{T}(g)|\leq (2q-3)$. Thus, the lemma is true for $s=2$. Now, we assume that $s \geq 3$. We consider the following two cases:
\begin{itemize}
\item If $g(\alpha,t_{2},\cdots,t_{s})=0$ for some $\alpha \in \mathbb{F}_{q}^{*}$, then $g=(t_{1}-\alpha)h+h_{1}$ where no term of $h_{1}$ is divisible by $t_{1}$. Putting $t_{1}=\alpha$, we get that $h_{1}$ is the zero polynomial. Thus, $g=(t_{1}-\alpha)h$. If deg $g=1$, then $h \in \mathbb{F}_{q}^{*}$ and $$| V_{T}(g)|=(q-1)^{s-1} \leq (q-1)^{s}-(q-2)^{d}(q-1)^{s-d}.$$ Therefore, we assume that deg $g \geq 2$.  By Lemma $3.1$, $h$ is square-free polynomial in $t_{2},~t_{3},\cdots,t_{s}$ and also we have deg $h=d-1$.  Let $T':=(\mathbb{F}_{q}^{*})^{s-1}$. Then, by induction hypothesis
\begin{align*}
|V_{T}(g)| &=|V_{T}(t_{1}-\alpha)|+|V_{T}(h)|-| V_{T}(t_{1}-\alpha) \cap V_{T}(h)| \\
&=(q-1)^{s-1}+(q-2)|V_{T'}(h)|\\
&\leq (q-1)^{s-1}+(q-2)[(q-1)^{s-1}-(q-2)^{d-1}(q-1)^{s-d}]\\
&=(q-1)^{s-1}+(q-2)(q-1)^{s-1}-(q-2)^{d}(q-1)^{s-d}\\
&=(q-1)^{s}-(q-2)^{d}(q-1)^{s-d}.\\
\end{align*}
\item If $g(\alpha,t_{2},\cdots,t_{s}) \neq 0$ for any $\alpha \in \mathbb{F}_{q}^{*}$, then let $\mathbb{F}_{q}^{*}:=\{\beta_{1},~\beta_{2},\cdots,\beta_{q-1} \}$. For $1 \leq i \leq q-1$, define $g_{i}(t_{2},~t_{3},\cdots,t_{s}):=g(\beta_{i},t_{2},\cdots,t_{s})$. We have the following inclusion
$$V_{T}(g) \hookrightarrow \cup_{i=1}^{q-1} (\{\beta_{i}\} \times V_{T'}(g_{i})),~a \mapsto a.$$ Therefore $|V_{T}(g)|\leq \sum_{i=1}^{q-1} |V_{T'}(g_{i})| $. For each $i$, $1 \leq i \leq q-1$, we have the following cases.
\begin{enumerate}
\item If each term of degree $d$ in $g$ contains $t_{1}$, then $g_{i}$ is a square-free polynomial in $s-1$ variables of degree $d-1$.
\item If there exists a term of degree $d$ in $g$ not containing $t_{1}$, then $g_{i}$ is a square-free polynomial in $s-1$ variables of degree $d$.\\
\end{enumerate}
Now, if each $g_{i}$ is of type $(1)$, then 
\begin{align*}
|V_{T}(g)|&\leq \sum_{i=1}^{q-1} |V_{T'}(g_{i})|\\
& \leq (q-1)[(q-1)^{s-1}-(q-2)^{d-1}(q-1)^{s-d}]\\
&=(q-1)^{s}-(q-2)^{d-1}(q-1)^{s-d+1}\\
&\leq (q-1)^{s}-(q-2)^{d}(q-1)^{s-d},
\end{align*}
as $(q-2)^{d-1}(q-1)^{s-d+1} \geq (q-2)^{d}(q-1)^{s-d}$. But if there exists atleast one $g_{i}$ of type $(2)$, then using the fact that $(q-1)^{s}-(q-2)^{d}(q-1)^{s-d} \geq (q-1)^{s}-(q-2)^{d'}(q-1)^{s-d'},\text{ for } d>d',$ we have 
\begin{align*}
|V_{T}(g)| &\leq \sum_{i=1}^{q-1} |V_{T'}(g_{i})| \\
& \leq (q-1)[(q-1)^{s-1}-(q-2)^{d}(q-1)^{s-d-1}]\\
&=(q-1)^{s}-(q-2)^{d}(q-1)^{s-d}.\\
\end{align*}
\end{itemize}
\end{proof}

\begin{lemma}
For $s \geq 2$ and $d < s$, let $f$ and $g$ be two distinct non-zero square-free polynomials of degree $d$ in $\mathbb{F}_{q}[t_{1},t_{2},\cdots,t_{s}]$. Then $$| V_{T}(f) \cap V_{T}(g)|\leq (q-1)^{s}-q(q-2)^{d}(q-1)^{s-d-1}.$$
\end{lemma}
\begin{proof}
We prove this lemma by induction on $s$. When $s=2$, $d=1$ and we have to show that for any two distinct square-free polynomials $f$ and $g$ in two variables of degree one, we have $|V_{T}(f) \cap V_{T}(g)| \leq 1$. For $\lambda, \mu, \delta, \alpha, \beta \in \mathbb{F}_{q}^{*}$, we obtain the following table by direct calculations. \\
\begin{center}
\begin{tabular}{ |l|l|l| } 
 \hline
  $f$ & $g$ & $| V_{T}(f) \cap V_{T}(g)|$\\ 
 \hline
 $\lambda t_{1}+\mu$ & $\alpha t_{1}$ & $0$\\
 $\lambda t_{1}+\mu$ & $\alpha t_{2}$ & $0$\\
 $\lambda t_{1}+\mu$ & $\alpha t_{2}+\beta$ & $1$\\
 $\lambda t_{1}+\mu$ & $\alpha t_{1}+\beta t_{2}$ & $1$\\
 $\lambda t_{1}+\mu$ & $\alpha t_{1}+\beta t_{2}+\delta$ & $\leq 1$\\
 $\lambda t_{1}$ & $g$ & $0$\\
 $\lambda t_{2}$ & $g$ & $0$\\
 $\lambda t_{2}+\mu$ & $\alpha t_{1}+\beta t_{2}$ & $1$\\
 $\lambda t_{2}+\mu$ & $\alpha t_{1}+\beta t_{2}+\delta$ & $\leq 1$\\
 $\lambda t_{1}+\mu t_{2}$ & $\alpha t_{1}+\beta t_{2}+\delta$ & $\leq 1$\\
 \hline
\end{tabular}
\end{center}
From the above table, we have $|V_{T}(f) \cap V_{T}(g)| \leq 1$. Thus, the lemma is true for $s=2$. So, we assume that $s \geq 3$. Let $T':=(\mathbb{F}_{q}^{*})^{s-1}$. To prove the lemma we consider the following cases.
\begin{itemize}
\item If $f=(t_{1}-\alpha)f_{1}$ and $g=(t_{1}-\alpha)g_{1}$ for some $\alpha \in \mathbb{F}_{q}^{*}$. Note that if $d=1$, then $f$ and $g$ are equal. So, we assume that $d \geq 2$. Now by induction hypothesis
\begin{align*}
|V_{T}(f) \cap V_{T}(g)| &=|V_{T}(t_{1}-\alpha)| +| V_{T}(f_{1}) \cap V_{T}(g_{1})| - |V_{T}(t_{1}-\alpha) \cap V_{T}(f_{1}) \cap V_{T}(g_{1})|\\
&=(q-1)^{s-1}+(q-2) |V_{T'}(f_{1}) \cap V_{T'}(g_{1})|\\
& \leq (q-1)^{s-1}+(q-2)[(q-1)^{s-1}-q(q-2)^{d-1}(q-1)^{s-d-1}]\\
&=(q-1)^{s}-q(q-2)^{d}(q-1)^{s-d-1}.\\
\end{align*}

\item If $f=(t_{1}-\alpha)f_{1}$ and $g=(t_{1}-\beta)g_{1}$ for some $\alpha,\beta \in \mathbb{F}_{q}^{*}$ with $\alpha \neq \beta$. If $d=1$, then $f_{1}, g_{1} \in \mathbb{F}_{q}^{*}$ and $|V_{T}(f) \cap V_{T}(g)|=0 \leq (q-1)^{s}-q(q-2)^{d}(q-1)^{s-d-1}.$ So we assume that $d \geq 2$, then by inclusion-exclusion principle
\begin{align*}
|V_{T}(f) \cap V_{T}(g)|&=|(V_{T}(t_{1}-\alpha) \cup V_{T}(f_{1})) \cap (V_{T}(t_{1}-\beta) \cup V_{T}(g_{1})) |\\
&=|V_{T}(t_{1}-\alpha)\cap V_{T}(g_{1})| + | V_{T}(t_{1}-\beta)\cap V_{T}(f_{1})| + | V_{T}(f_{1})\cap V_{T}(g_{1})|\\
&\hspace{0.4 cm}-| V_{T}(t_{1}-\alpha)\cap V_{T}(g_{1}) \cap V_{T}(f_{1})|-| V_{T}(t_{1}-\beta)\cap V_{T}(f_{1}) \cap V_{T}(g_{1})|\\
&=|V_{T'}(g_{1})|+ | V_{T'}(f_{1})| +(q-3) |V_{T'}(f_{1}) \cap V_{T'}(g_{1})| \\
& \leq 2[(q-1)^{s-1}-(q-2)^{d-1}(q-1)^{s-d}]+(q-3)[(q-1)^{s-1}\\
&\hspace{7 cm}-q(q-2)^{d-1}(q-1)^{s-d-1}]\\
&=(q-1)^{s}-(q-2)^{d-1}(q-1)^{s-d-1}(q^{2}-q-2)\\
&\leq (q-1)^{s}-q(q-2)^{d}(q-1)^{s-d-1},
\end{align*}
as $(q-2)^{d-1}(q-1)^{s-d-1}(q^{2}-q-2) \geq q(q-2)^{d}(q-1)^{s-d-1}$.\\

\item If $(t_{1}-\alpha) \nmid f$ for any $\alpha \in \mathbb{F}_{q}^{*}$ but $g=(t_{1}-\beta)g_{1}$ for some $\beta \in \mathbb{F}_{q}^{*}$, then let $\mathbb{F}_{q}^{*}:=\{ \beta_{1},~\beta_{2},\cdots, \beta_{q-1}\}$ and for each i, $1 \leq i \leq q-1$, set $f_{i}(t_{2},~t_{3},\cdots,t_{s}):=f(\beta_{i},t_{2},\cdots t_{s})$ and $\beta=\beta_{j}$ for some $1 \leq j \leq q-1$. Thus, we have
\begin{align*}
| V_{T}(f) \cap V_{T}(g)| &\leq | (\cup_{i=1}^{q-1} V_{T'}(f_{i})) \cap V_{T}((t_{1}-\beta)g_{1})| \\
&=| (\cup_{i=1}^{q-1} V_{T'}(f_{i})) \cap V_{T}(t_{1}-\beta)| + | (\cup_{i=1}^{q-1} V_{T'}(f_{i})) \cap V_{T}(g_{1})|\\
&\hspace{3.5 cm}-|(\cup_{i=1}^{q-1} V_{T'}(f_{i})) \cap V_{T}(t_{1}-\beta) \cap V_{T}(g_{1})|\\
&=|V_{T'}(f_{j})| +\sum_{i=1}^{q-1} | V_{T'}(f_{i}) \cap V_{T'}(g_{1})| -| V_{T'}(f_{j}) \cap V_{T_{1}}(g_{1})|.\\
&\leq | V_{T'}(f_{j})| +(q-2)|V_{T'}(g_{1})|.\\
& \leq (q-1)^{s-1}-(q-2)^{d}(q-1)^{s-d-1}+(q-2)[(q-1)^{s-1}\\
&\hspace{6 cm}-(q-2)^{d-1}(q-1)^{s-d}]\\
&=(q-1)^{s}-(q-2)^{d}(q-1)^{s-d-1}[q-1+1]\\
&=(q-1)^{s}-q(q-2)^{d}(q-1)^{s-d-1}.\\
\end{align*}
\item If $(t_{1}-\alpha) \nmid f$ and $(t_{1}-\alpha) \nmid g$ for any $\alpha \in \mathbb{F}_{q}^{*}$, then for $1 \leq i \leq q-1$, set $f_{i}(t_{2}, \cdots,t_{s}):=f(\beta_{i},t_{2},\cdots,t_{s})$ and $g_{i}(t_{2}, \cdots, t_{s}):=g(\beta_{i},t_{2},\cdots,t_{s})$. Thus, we have
\begin{equation}
|V_{T}(f) \cap V_{T}(g)| \leq \sum_{i=1}^{q-1} | V_{T'}(f_{i}) \cup V_{T'}(g_{i})|.
\end{equation}

Now for each $i$, $1 \leq i \leq q-1$, we have the following cases.\\
\begin{enumerate}
\item If $deg~f_{i}=deg~g_{i}=d-1.$ If $d=1$ then $f_{i},g_{i} \in \mathbb{F}_{q}^{*}$ and $|V_{T'}(f_{i})\cap V_{T'}(g_{i})|=0$. So, we assume $d \geq 2$. Then by induction hypothesis
\begin{align*}
|V_{T'}(f_{i})\cap V_{T'}(g_{i})| &\leq (q-1)^{s-1}-q(q-2)^{d-1}(q-1)^{s-d-1}.\\
\end{align*}

\item If one of $f_{i}$ and $g_{i}$ has degree $d-1$. Let us call it $g'_{i}$, then by Lemma $3.2$
\begin{align*}
|V_{T'}(f_{i})\cap V_{T'}(g_{i})| &\leq |V_{T'}(g'_{i}) \mid\\
& \leq (q-1)^{s-1}-(q-2)^{d-1}(q-1)^{s-d}.\\
\end{align*}

\item If $deg~f_{i}=deg~g_{i}=d$ and $d<s-1$, then by induction hypothesis
\begin{align*}
|V_{T'}(f_{i})\cap V_{T'}(g_{i})| &\leq (q-1)^{s-1}-q(q-2)^{d}(q-1)^{s-d-2}.\\
\end{align*}

\item If $deg~f_{i}=deg~g_{i}=d$ and $d=s-1$, then $f_{i}$ and $g_{i}$ are square-free polynomials of degree $d$ in $d$ variables so the leading monomial of $f_{i}$ and $g_{i}$ are equal. We construct square-free polynomials $f'_{i}$ and $g'_{i}$ as follows.
$$f'_{i}:=f_{i} \text{ and } g'_{i}:=f_{i}-\frac{LC(f_{i})}{LC(g_{i})} g_{i}.$$
Then, $g'_{i}$ is a square-free polynomial of degree $d' \leq d-1$ in $s-1$ variables. Also, $V_{T'}(f_{i},g_{i})=V_{T'}(f'_{i},g'_{i})$. Then
\begin{align*}
|V_{T'}(f_{i})\cap V_{T'}(g_{i})| & \leq |V_{T'}(g'_{i})|\\
&\leq (q-1)^{s-1}-(q-2)^{d'}(q-1)^{s-d'-1}\\
&\leq (q-1)^{s-1}-(q-2)^{d-1}(q-1)^{s-d}.\\
\end{align*}
\end{enumerate}
Therefore, in equation $(3.2)$ we have if $d<s-1$ then 
\begin{align*}
|V_{T}(f) \cap V_{T}(g)| &\leq \sum_{i=1}^{q-1} | V_{T'}(f_{i}) \cup V_{T'}(g_{i})|\\
&\leq (q-1)[(q-1)^{s-1}-q(q-2)^{d}(q-1)^{s-d-2}]\\
&\leq (q-1)^{s}-q(q-2)^{d}(q-1)^{s-d-1}.
\end{align*}
But if $d=s-1$ then from cases $(1),(2)$ and $(4)$, we get 
\begin{align*}
|V_{T}(f) \cap V_{T}(g)| &\leq \sum_{i=1}^{q-1} | V_{T'}(f_{i}) \cup V_{T'}(g_{i})|\\
&\leq (q-1)[(q-1)^{s-1}-(q-2)^{s-2}(q-1)]\\
&\leq (q-1)^{s}-(q-2)^{s-2}(q-1)^{2}\\
&\leq (q-1)^{s}-q(q-2)^{s-1}=(q-1)^{s}-q(q-2)^{d}(q-1)^{s-d-1}.
\end{align*}
\end{itemize}
\end{proof}

Extending Lemma $3.2$ and $3.3$ for three or more square-free polynomials in $S$ requires dealing with many cases and is tiresome. Also, these lemmas calculate the cardinality of the sets when all the polynomials are of degree $d$. In the remaining part of this section, we give an upper bound on the number of zeroes in the affine torus $T=(\mathbb{F}_{q}^{*})^{s}$ of square-free polynomials of degree $d$. Then we obtain an upper bound on the number of zeroes in $T$ of square-free polynomials of degree at most $d$. \\
 
Let $s\geq 2$, $1 \leq d \leq s$ be fixed and $1 \leq r \leq {s \choose d}$. Let $f_{1},f_{2},\cdots,f_{r} \in KV_{\leq d}$ be linearly independent polynomials of degree $d$. We assume that their leading monomials are distinct. To calculate the number of zeroes of $f_{1},f_{2},\cdots,f_{r}$ in $T$ we use the ideas of \cite{acc}.\par

We have $I=I(T)$ is the vanishing ideal of $T$ in $S$. The set $\{ t_{i}^{q-1}-1 ~:~ i=1, \cdots, s \}$ is a Groebner basis of $I$. The ideal $L:=\langle LT(I)\rangle$ is generated by the set $\{t_{i}^{q-1} ~:~ i=1, \cdots, s \}.$ Let $J:=LT(I(V_{T}(f_{1},~f_{2},\cdots,f_{r})))$ and for $i=1,~2, \cdots, r$, let $t^{a_{i}}=t_{1}^{a_{i,1}} t_{2}^{a_{i,2}}\cdots t_{s}^{a_{i,s}}:=LT(f_{i})$. Consider the ideal $\mathcal{A}:=\langle t_{1}^{q-1},~t_{2}^{q-1},\cdots, t_{s}^{q-1}, t^{a_{1}},~t^{a_{2}},\cdots, t^{a_{r}} \rangle$. From Proposition $2.7$ and $2.8$  we have  \begin{center}$| V_{T}(f_{1},\cdots,f_{r})|=$ $^{a} \mathrm{{HF}_{J}}(u)\leq$ $^{a}\mathrm{{HF}_{\mathcal{A}}}(u),$\\\end{center}
for all sufficiently large $u$. Thus, our next goal is to calculate $^{a}\mathrm{{HF}_{\mathcal{A}}}(u)$. Before that, we introduce the following notations as in \cite{acc}.
\begin{definition}
Let $k:=s(q-2)$.
\begin{enumerate}
\item $F:=(\{0,1,\cdots, q-2\})^{s}$ and $G:=(\{0,1\})^{s}$.
\item For $b:=(b_{1},~b_{2},\cdots,b_{s}) \in F$, define $deg(b)=b_{1}+~b_{2}\cdots+b_{s}$.
\item For $u \leq k$, define $F_{u}:=\{b \in F:~ deg(b)=u\}$ and  $F_{\leq u}:=\{b \in F:~ deg(b)\leq u\}$.
\item $(b_{1},~b_{2},\cdots,b_{s}) \leq_{P} (c_{1},~c_{2},\cdots,c_{s})$ if and only if $b_{1} \leq c_{1},~b_{2} \leq c_{2}, \cdots, b_{s} \leq c_{s}$.
\item For $H \subseteq F$, define shadow of $H$ as
$${\nabla}_{F} (H):=\{a \in F : b \leq_{P} a \text{ for some } b \in H\}.$$
\end{enumerate}
\end{definition}


Following the idea as in \cite{acc}, we write $\mathcal{A}$ as $\mathcal{A}=\mathcal{A}_{1}+\mathcal{A}_{2}$ where $\mathcal{A}_{1}=\langle t_{1}^{q-1},~t_{2}^{q-1}, \cdots, t_{s}^{q-1} \rangle$ and $\mathcal{A}_{2}=\langle t^{a_{1}},~t^{a_{2}}, \cdots, t^{a_{s}} \rangle$. Then any monomial $t^{b}:=t_{1}^{b_{1}}t_{2}^{b_{2}}\cdots t_{s}^{b_{s}}$ that doesn't belong to $\mathcal{A}_{1}$ has $b_{i} \leq q-2$, for all $i$, $1 \leq i\leq s$. Now, if $M_{\mathcal{A}_{1}}$ denotes the set of monomials that does not belong to $\mathcal{A}_{1}$, then $M_{\mathcal{A}_{1}}$ is in bijection with the set $F$. From Proposition $2.7$, $t^{b} \in M_{\mathcal{A}_{1}}$ will belong to $\mathcal{A}_{2}$ if and only if $t^{a_{j}} \mid t^{b}$ for some $j$, $1 \leq j \leq r$ i.e.\ $a_{j}\leq_{P} b$. Also, $\{a_{1},~a_{2},\cdots ,a_{r}\} \subseteq G_{d}$. Thus, we have
$${^{a} \mathrm{HF}_{\mathcal{A}}}(u)=| F \backslash \nabla_{F}(a_{1},\cdots,a_{r})|,$$ where $u \geq k$. Hence,
\begin{equation}
|V_{T}(f_{1},\cdots f_{r})|\leq max~\{~| F \backslash \nabla_{F}(a_{1},~a_{2},\cdots,a_{r})|~ :~ a_{1},~a_{2}, \cdots, a_{r} \in G_{d}\}.\\
\end{equation}

The following lemma gives a lower bound on $| \nabla_{F}(\{a_{1},~a_{2},\cdots,a_{r}\})|$ for $a_{1},~a_{2},\cdots,a_{r} \in G_{d}$.
\begin{lemma}
Let $s\geq 2$, $1 \leq d \leq s$ and $1\leq r \leq {s \choose d}$. If $d+r-2<s$ then for any $B=\{a_{1},~a_{2},\cdots,a_{r}\}\subseteq G_{d}$ with $|B|=r$, we have 
$$|\nabla_{F}(B)|~ \geq (q-2)^{d-1} (q-1)^{s-d-r+1}[(q-1)^{r}-1].$$
\end{lemma}
\begin{proof}
We prove this lemma by induction on $r$. For $r=1$, the lemma is clearly true. For $r=2$, let $B=\{a,b\} \subseteq G_{d}$ where $a=(a_{1},a_{2},\cdots,a_{s})$ and $b=(b_{1},b_{2},\cdots,b_{s})$. Define for any $v=(v_{1},~v_{2},\cdots,v_{s}) \in F$, $$supp~v:=\{1 \leq i \leq s~:~v_{i}\neq 0\}.$$ Let $A:=supp~a \cup supp~b$ and $|A|=:e$. Then 
\begin{align*}
|\nabla_{F}(B)\mid &=\prod_{i=1}^{s}(q-1-a_{i})+\prod_{i=1}^{s}(q-1-b_{i})-\prod_{i=1}^{s} min(q-1-a_{i},q-1-b_{i})\\
&=2(q-2)^d(q-1)^{s-d}-(q-2)^{e}(q-1)^{s-e}.
\end{align*}
To prove the lemma for $r=2$, we have to show that 
$$2(q-2)^d(q-1)^{s-d}-(q-2)^{e}(q-1)^{s-e}\geq q(q-2)^{d}(q-1)^{s-d-1},$$
which is equivalent to proving that 
\begin{equation}
(q-2)^{d+1}(q-1)^{s-d-1}\geq(q-2)^{e}(q-1)^{s-e}.
\end{equation}
Observe that $e \geq d+1$, so equation $(3.4)$ holds.\\

Assume that the lemma is true for $r-1$. We prove it for $r$. For any $B=\{a_{1},\cdots,a_{r}\} \subseteq G_{d}$,
$$\nabla_{F}(B)=\nabla_{F}(\{a_{1},a_{2},\cdots,a_{r}\})=\nabla_{F}(\{a_{1},a_{2},\cdots,a_{r-1}\}) \cup [\nabla (\{ a_{r} \})\setminus \nabla_{F}(\{a_{1},a_{2},\cdots,a_{r-1}\})].$$

Therefore
\begin{equation}
 |\nabla_{F}(B)|=|\nabla_{F}(\{a_{1},a_{2},\cdots,a_{r-1}\})|+ |\nabla (\{ a_{r} \})\setminus \nabla_{F}(\{a_{1},a_{2},\cdots,a_{r-1}\})|.
 \end{equation}

We arrange $a_{1},a_{2},\cdots, a_{r}$ in graded lexicographic order. Without loss of generality assume that $a_{r} \prec \cdots \prec a_{2} \prec a_{1}$.\par

Now, look at $a_{1}$, find a position $p_{1}$ $(1 \leq p_{1} \leq s)$ such that any $b \in F$ with $b_{p_{1}}=0$ doesn't belong to $\nabla_{F}(\{a_{1}\})$ but belongs to $\nabla_{F}(\{a_{r}\})$. (It may or may not belong to $\nabla_{F}(\{a_{2},\cdots,a_{r-1}\})$. Next, look at $a_{2}$ and find a position $p_{2}$ similarly ( Note that $p_{2}$ may be equal to $p_{1}$). Keep on doing this upto $a_{r-1}$. In this way, we get $p_{1}, p_{2}, \cdots, p_{r-1}$. Let $e_{1}, e_{2}, \cdots, e_{w}$ denotes distinct elements from $p_{1}, p_{2}, \cdots, p_{r-1}$. Then $1 \leq w \leq r-1$. Let $1 \leq v_{1},~v_{2},\cdots,v_{d}\leq s$ be the positions where $a_{r}$ is non-zero.\par  Consider $b=(b_{1},~b_{2},\cdots,b_{s}) \in F$ with $b_{e_{i}}=0$ for $1 \leq i \leq w$ and $b_{v_{j}}=1$ for $1\leq j \leq d.$ Then, any such $b$ is contained in $\nabla_{F}(\{a_{r}\})$ but not in $\nabla_{F}(\{a_{1},\cdots,a_{r-1}\})$. The cardinality of set of all such $b$'s is $(q-2)^{d}(q-1)^{s-w-d} \geq (q-2)^{d}(q-1)^{s-d-r+1}$.\\

By induction hypothesis, we obtain from equation $(3.5)$
\begin{align*}
|\nabla_{F}(B)| &\geq (q-2)^{d-1}(q-1)^{s-d-r+2}[(q-1)^{r-1}-1]+(q-2)^{d} (q-1)^{s-d-r+1}\\
&=(q-2)^{d-1}(q-1)^{s-d-r+1}[(q-1)^{r}-1].\\
\end{align*}
\end{proof}

Now, we state a lemma which will be required in our main result. The proof is similar to \cite{hyper}, Lemma $3.2$.
\begin{lemma}
Let $\mathcal{L}$ be a $K$-linear subspace of $S=K[t_{1}, \cdots,t_{s}]$ of finite dimension and let $F'=\{f_{1},~f_{2},\cdots,f_{r}\}$ be a subset of $\mathcal{L} \backslash \{0\}$. If $f_{1},~f_{2},\cdots,f_{r}$ are linearly independent over $K$, then there is a set $G'=\{g_{1},~g_{2},\cdots,g_{r}\} \subset \mathcal{L}  \backslash \{0\}$ such that
\begin{itemize}
\item $KF'=KG'$.
\item $LM(g_{1}),~LM(g_{2}),\cdots, LM(g_{r})$ are distinct.
\item $LM(g_{i}) \preceq LM(f_{i})$ for all $i$.
\item $g_{1},~g_{2}, \cdots, g_{r}$ are linearly independent over $K$.
\item $V_{T}(F')=V_{T}(G')$.
\end{itemize}    
\end{lemma}
\begin{proof}
We proceed by induction on $r$. The case $r=1$ is clear. Let $r=2$. For $F'=\{f_{1},f_{2}\} \subset \mathcal{L} \backslash \{0\}$ linearly independent set over $K$. If $LM(f_{1})\neq LM(f_{2})$, then $G':=F'$ works. Otherwise, define\\
 $$g_{1}:=f_{1} \text{ and } g_{2}:=f_{1}-\frac{LC(f_{1})}{LC(f_{2})} f_{2}.$$ Then $G':=\{g_{1},g_{2}\} \subset \mathcal{L}  \backslash \{0\}$ such that $KF'=KG'$, $LM(g_{1}) \neq LM(g_{2})$ and $LM(g_{i}) \preceq LM(f_{i})$ for $i=1,2$. Also, $G'$ is linearly independent set over $K$ and $V_{T}(F')=V_{T}(G')$.\par
Now, assume that $r>2$ and $LM(f_{r}) \preceq \cdots \preceq LM(f_{2}) \preceq LM(f_{1})$. We have the following two cases.
\begin{itemize}
\item If $LM(f_{2}) \prec LM(f_{1})$, then applying induction hypothesis to the set $F''=\{f_{2},~f_{3},\cdots,f_{r}\}$ we obtain a set $G''=\{g_{2},~g_{3},\cdots,g_{r}\} \subset \mathcal{L}  \backslash \{0\}$ such that $KF''=KG''$, $LM(g_{2}),~LM(g_{2})$ $\cdots, LM(g_{r})$ are distinct, $LM(g_{i}) \preceq LM(f_{i})$ for $i=2,~3,\cdots, r$ and $g_{2},~g_{3}, \cdots, g_{r}$ are linearly independent over $K$. Also $V_{T}(F'')=V_{T}(G'')$. Define $g_{1}:=f_{1}$ and $G':=G'' \cup \{g_{1}\}$. This implies $KF'=KG'$. Since $LM(g_{i}) \preceq LM(f_{i}) \prec LM(f_{1})$ for $i=2,3,\cdots, r$, the monomials $LM(g_{i})$, $1 \leq i \leq r$, are distinct. Also, $G'$ is linearly independent over $K$ and
 $$V_{T}(F')=V_{T}(f_{1}) \cap V_{T}(F'')=V_{T}(g_{1}) \cap V_{T}(G'')=V_{T}(G').$$

\item If $LM(f_{2})=LM(f_{1})$, assume that there exists $l\geq 2$ such that $LM(f_{1})=LM(f_{i})$ for $i \leq l$ and $LM(f_{i}) \prec LM(f_{1})$ for $i>l$. Define $$h_{i}=f_{1}-\frac{LC(f_{1})}{LC(f_{i})} f_{i} \text{ for } i=2,~3, \cdots, l \text{ and } h_{i}=f_{i} \text{ for } i>l.$$ Then $LM(h_{i}) \prec LM(f_{1})$ for $i \geq 2$ and $H=\{h_{2},~h_{3},\cdots,h_{r}\} \subset \mathcal{L} \backslash \{0\}$ is a linearly independent set. By induction hypothesis for $H$, we obtain a set $G''=\{g_{2},~g_{3},\cdots, g_{r}\} \subset \mathcal{L}\backslash \{0\}$ such that $KH=KG''$, $LM(g_{i}), i=2,~3,\cdots,r$ are distinct and $LM(g_{i}) \preceq LM(h_{i})$ for $i=2,~3,\cdots,r $. Also,  $G''$ is linearly independent set over $K$ and $V_{T}(H)=V_{T}(G'')$. Define $g_{1}:=f_{1}$ and $G':=G'' \cup \{g_{1}\}$. We obtain $G' \subset \mathcal{L}  \backslash \{0\}$ such that $LM(g_{1}),~LM(g_{2}), \cdots, LM(g_{r})$ are distinct. As $LM(g_{i}) \preceq LM(h_{i}) \prec LM(f_{1})$ for $i=2,~3, \cdots,r$, we have $LM(g_{i}) \preceq LM(f_{i})$ for $i=1,~2,\cdots,r$. Also, $G'$ is linearly independent set. Thus, as $$V_{T}(F')=V_{T}(f_{1}) \cap V_{T}(H)=V_{T}(g_{1}) \cap V_{T}(G'')=V_{T}(G').$$
\end{itemize}
\end{proof}

Thus, we get our main result.
\begin{theorem}
Let $s\geq 2$, $1 \leq d \leq s$ and $1 \leq r \leq {s \choose d}$. If $d+r-2<s$, then for any $f_{1},f_{2},\cdots,f_{r} \in KV_{\leq d}$  of linearly independent polynomials  over $K$ of degree $d$, we have 
$$|V_{T}(f_{1},f_{2},\cdots,f_{r})| \leq (q-1)^{s}-(q-2)^{d-1}(q-1)^{s-d-r+1}[(q-1)^{r}-1].$$
\end{theorem}
\begin{proof}
Suppose leading monomials of $f_{1}, f_{2}, \cdots, f_{r}$ are distinct, then the inequality follows from equation $(3.3)$ and Lemma $3.5$. If the leading monomials of $f_{1}, f_{2}, \cdots, f_{r}$  are not all distinct, then by Lemma $3.6$, we get $g_{1},~g_{2},\cdots,g_{r}$ with distinct leading monomials. If $deg~g_{i}=d$ for all $i$. Then the theorem follows from equation $(3.3)$ and Lemma $3.5$. If atleast one $g_{i}$ has degree less than $d$, without loss of generality assume that $deg~g_{r}=d'<d$. Then, 
\begin{align*}
|V_{T}(f_{1},f_{2},\cdots,f_{r})|&=|V_{T}(g_{1},g_{2},\cdots,g_{r})|\\
&\leq |V_{T}(g_{r})|\\
&\leq (q-1)^{s}-(q-2)^{d'}(q-1)^{s-d'}\\
&\leq (q-1)^{s}-(q-2)^{d-1}(q-1)^{s-d-r+1}[(q-1)^{r}-1].
\end{align*}
\end{proof}

Now,  Let $f_{1},f_{2},\cdots,f_{r} \in KV_{\leq d}$ be linearly independent. Since the polynomials are linearly independent, we can assume that their leading monomials are distinct. Following the procedure as before, we get 
\begin{equation}
|V_{T}(f_{1},\cdots f_{r})|\leq max~\{~| F \backslash \nabla_{F}(a_{1},~a_{2},\cdots,a_{r})|~ :~ a_{1},~a_{2}, \cdots, a_{r} \in G_{\leq d}\}.\\
\end{equation}

Repeating the procedure of Lemma $3.5$, we have the following lemma.
\begin{lemma}
Let $s\geq 2$, $1 \leq d \leq s$ and $1\leq r \leq dim~KV_{\leq d}$. If $d+r-2<s$, then for any $B=\{a_{1},~a_{2},\cdots,a_{r}\}\subseteq G_{\leq d}$ with $|B|=r$, we have 
$$|\nabla_{F}(B)|~ \geq (q-2)^{d-1} (q-1)^{s-d-r+1}[(q-1)^{r}-1].$$
\end{lemma}
\begin{proof}
We prove this lemma by induction on $r$. For $r=1$, the inequality holds. For $r=2$, let $B=\{a,b\} \subseteq G_{\leq d}$ where $a=(a_{1},a_{2},\cdots,a_{s})$ and $b=(b_{1},b_{2},\cdots,b_{s})$. Let $A_{1}:=supp~a$, $A_{2}:=supp~b$ and let $|A_{1}|=:e_{1}$, $|A_{2}|=:e_{2}$. Without loss of generality, suppose that $e_{1} \leq e_{2} \leq d$. Then 
\begin{align*}
|\nabla_{F}(B)\mid &=\prod_{i=1}^{s}(q-1-a_{i})+\prod_{i=1}^{s}(q-1-b_{i})-\prod_{i=1}^{s} min(q-1-a_{i},q-1-b_{i})\\
&=(q-2)^{e_{1}}(q-1)^{s- e_{1}}+(q-2)^{e_{2}}(q-1)^{s- e_{2}}-(q-2)^{|A_{1} \cup A_{2}|}(q-1)^{s-|A_{1} \cup A_{2}|}.
\end{align*}
To prove the lemma for $r=2$, we consider the following two cases as in \cite{hyper}.
\begin{itemize}
\item If $e_{1}=d$, then we have to show 
$$(q-2)^{d+1}(q-1)^{s-d-1} \geq (q-2)^{|A_{1} \cup A_{2}|}(q-1)^{s-|A_{1} \cup A_{2}|},$$
which we have already proved in previous lemma.

\item If $e_{1}<d$, then we the lemma holds true using the following inequalities $$(q-2)^{e_{1}}(q-1)^{s- e_{1}} \geq q(q-2)^{d}(q-1)^{s-d-1}$$ and \hspace{ 1 cm} $(q-2)^{e_{2}}(q-1)^{s- e_{2}} \geq (q-2)^{|A_{1} \cup A_{2}|}(q-1)^{s-|A_{1} \cup A_{2}|}$.
\end{itemize}

Assume that the lemma is true for $r-1$. We prove for $r$. For any $B=\{a_{1},~a_{2},\cdots,a_{r}\} \subseteq G_{\leq d}$,
$$\nabla_{F}(B)=\nabla_{F}(\{a_{1},a_{2},\cdots,a_{r}\})=\nabla_{F}(\{a_{1},a_{2},\cdots,a_{r-1}\}) \cup [\nabla (\{ a_{r} \})\setminus \nabla_{F}(\{a_{1},a_{2},\cdots,a_{r-1}\})].$$

Therefore,
\begin{equation}
 |\nabla_{F}(B)|=|\nabla_{F}(\{a_{1},a_{2},\cdots,a_{r-1}\})|+ |\nabla (\{ a_{r} \})\setminus \nabla_{F}(\{a_{1},a_{2},\cdots,a_{r-1}\})|.
 \end{equation}

We arrange $a_{1},a_{2},\cdots, a_{r}$ in graded lexicographic order. Without loss of generality, assume that $a_{r} \prec \cdots \prec a_{2} \prec a_{1}$.\par

Now, look at $a_{1}$, find a position $p_{1}$ $(1 \leq p_{1} \leq s)$ such that any $b \in F$ with $b_{p_{1}}=0$ doesn't belong to $\nabla_{F}(\{a_{1}\})$ but belongs to $\nabla_{F}(\{a_{r}\})$. (It may or may not belong to $\nabla_{F}(\{a_{2},\cdots,a_{r-1}\})$. Next, look at $a_{2}$ and find a position $p_{2}$ similarly ( Note that $p_{2}$ may be equal to $p_{1}$). Keep on doing this upto $a_{r-1}$. In this way, we get $p_{1}, p_{2}, \cdots, p_{r-1}$. Let $e_{1}, e_{2}, \cdots, e_{w}$ denotes distinct elements from $p_{1}, p_{2}, \cdots, p_{r-1}$. Then $1 \leq w \leq r-1$. Let $1 \leq v_{1},~v_{2},\cdots,v_{l}\leq s$ be the positions where $a_{r}$ is non-zero. Note that $l \leq d$. If $l=0$, then $a_{r}=(0,\cdots,0)$ and $|\nabla_{F}(B)|=(q-1)^{s} \geq (q-2)^{d-1}(q-1)^{s-d-r+1}[(q-1)^{r}-1]$. So, we assume that $l \geq 1$. \par  Consider $b=(b_{1},~b_{2},\cdots,b_{s}) \in F$ with $b_{e_{i}}=0$ for $1 \leq i \leq w$ and $b_{v_{j}}=1$ for $1\leq j \leq l.$ Then any such $b$ is contained in $\nabla_{F}(\{a_{r}\})$ but not in $\nabla_{F}(\{a_{1},\cdots,a_{r-1}\})$. The cardinality of set of all such $b$'s is $(q-2)^{l}(q-1)^{s-w-l} \geq (q-2)^{d}(q-1)^{s-d-r+1}$.\\

By induction hypothesis, we obtain from equation $(3.7)$
\begin{align*}
|\nabla_{F}(B)| &\geq (q-2)^{d-1}(q-1)^{s-d-r+2}[(q-1)^{r-1}-1]+(q-2)^{d} (q-1)^{s-d-r+1}\\
&=(q-2)^{d-1}(q-1)^{s-d-r+1}[(q-1)^{r}-1].
\end{align*}
\end{proof}

Thus, we get
\begin{theorem}
For $1 \leq r \leq dim~KV_{\leq d}$, if $f_{1},\cdots,f_{r} \in KV_{\leq d}$ are linearly independent and $d+r-2<s$, then $$| V_{T}(f_{1},\cdots,f_{r})| \leq (q-1)^{s}-(q-2)^{d-1}(q-1)^{s-d-r+1}[(q-1)^{r}-1].$$
\end{theorem} 
\section{Generalized Hamming weights of certain evaluation codes}
Let $s$ and $d$ be integers such that $s\geq 2$ and $1 \leq d \leq s$. In this section, we determine the generalized Hamming weights of toric codes $C_{d}$ and $C^{\mathbb{P}}_{d}$ over hypersimplices, as defined in section $2.1$. We also determine the generalized Hamming weights of square-free affine evaluation code $C_{\leq d}$, as defined in section $2.2$.

\subsection{Generalized Hamming weights of $C_{d}$ and $C^{\mathbb{P}}_{d}$}
For $1 \leq r \leq dim_{K} KV_{d}$, the $r$-th generalized Hamming weight $d_{r}(C^{\mathbb{P}}_{d})$ of $C^{\mathbb{P}}_{d}$ is given by\\
$$d_{r}(C^{\mathbb{P}}_{d}):=min\{~|\mathbb{T}|\backslash |V_{\mathbb{T}}(H)|~:~H:=\{f_{1},~f_{2},\cdots, f_{r}\}  \subseteq KV_{d} \text{ is linearly independent over } K \}.$$
Similarly, we define $d_{r}(C_{d})$. (It follows from \cite{hyper} and \cite{genera}).

In this subsection, we find formulae for generalized Hamming weights of codes $C_{d}$ and $C^{\mathbb{P}}_{d}$ under certain cases. 

\begin{theorem}
Let $1 \leq r \leq {s \choose d}$. For $2d+r-2<s$, we have 
$$d_{r}(C^{\mathbb{P}}_{d})=(q-2)^{d-1}(q-1)^{s-d-r}[(q-1)^{r}-1].$$
\end{theorem} 
\begin{proof} For any $f_{1},~f_{2},\cdots, f_{r} \in KV_{d}$ linearly independent over $K$, we have
$$(q-1)|V_{\mathbb{T}}(f_{1},~f_{2},\cdots, f_{r})|=|V_{T}(f_{1},~f_{2},\cdots, f_{r})|.$$
So, from Theorem $3.7$ as $d+r-2<2d+r-2<s$, we get
$$|V_{\mathbb{T}}(f_{1},~f_{2},\cdots, f_{r})|\leq (q-1)^{s-1}-(q-2)^{d-1}(q-1)^{s-d-r}[(q-1)^{r}-1].$$ Thus,
$$d_{r}(C^{\mathbb{P}}_{d}))\geq (q-2)^{d-1}(q-1)^{s-d-r}[(q-1)^{r}-1].$$

For the converse, consider the polynomials $g_{1},~g_{2},\cdots,g_{r}$ where $$g_{i}:=(t_{1}-t_{2})(t_{3}-t_{4})\cdots (t_{2d-3}-t_{2d-2})(t_{2d+i-2}-t_{2d+i-1}), \text{ for } 1 \leq i \leq r.$$\\
Then $g_{1},~g_{2},\cdots, g_{r} \in KV_{d}$ and are linearly independent over $K$. Let $g:=(t_{1}-t_{2})(t_{3}-t_{4})\cdots (t_{2d-3}-t_{2d-2})$ and $h_{i}:=(t_{2d+i-2}-t_{2d+i-1})$ for $1 \leq i \leq r$. Let $T_{1}=(\mathbb{F}_{q}^{*})^{2d-2}$. Then 
\begin{align*}
|V_{T}(g_{1},g_{2}, \cdots,g_{r})| &=|V_{T}(g_{1})\cap V_{T}(g_{2})\cap \cdots \cap V_{T}(g_{r})|\\
&=|V_{T}(g) \cup (V_{T}(h_{1})\cap V_{T}( h_{2})\cap \cdots \cap V_{T}( h_{r}))|\\
&=| V_{T}(g) \mid + \mid V_{T}(h_{1})\cap V_{T}(h_{2})\cap \cdots \cap V_{T}(h_{r})|\\
&\hspace{5 cm}-|V_{T}(g) \cap V_{T}(h_{1})\cap V_{T}(h_{2})\cap \cdots \cap V_{T}(h_{r})|\\
&=(q-1)^{s-2d+2} |V_{T_{1}}(g)| +(q-1)^{s-r}-(q-1)^{s-2d-r+2} |V_{T_{1}}(g)|\\
&=(q-1)^{s-r}+(q-1)^{s-2d-r+2}[(q-1)^{r}-1] |V_{T_{1}}(g)|\\
&=(q-1)^{s-r}+(q-1)^{s-2d-r+2}[(q-1)^{r}-1]\left[ \sum_{i=1}^{d-1} (-1)^{i-1} { d-1 \choose i} (q-1)^{2d-2-i}\right]\\
&=(q-1)^{s}-(q-1)^{s-d-r+1}(q-2)^{d-1}[(q-1)^{r}-1]\\
&=(q-1)|V_{\mathbb{T}}(g_{1},g_{2}, \cdots,g_{r})|.\\
\end{align*}
Thus, $$d_{r}(C^{\mathbb{P}}_{d})\leq (q-2)^{d-1}(q-1)^{s-d-r}[(q-1)^{r}-1].$$
Hence, the result follows.
\end{proof}

\begin{corollary}
Let $1 \leq r \leq {s \choose d}$. For $2d+r-2<s$, we have 
$$d_{r}(C_{d})=(q-2)^{d-1}(q-1)^{s-d-r+1}[(q-1)^{r}-1],$$
\end{corollary}

Now, consider the following definition.
\begin{definition}
For a homogeneous polynomial $f \in S$ of degree $d$, we define the polynomial
$$f^{*}(t_{1},~t_{2},\cdots,t_{s}):=t_{1}t_{2}\cdots t_{s}f(t_{1}^{-1},~t_{2}^{-1},\cdots,t_{s}^{-1}).$$
\end{definition}
\vspace{-0.1 cm}
Then, $f^{*} \in S$ and is of degree $s-d$. Also, $(a_{1},~a_{2},\cdots,a_{s}) \in V_{T}(f)$ if and only if $(a_{1}^{-1},~a_{2}^{-1},\cdots,a_{s}^{-1}) \in V_{T}(f^*)$.\\\\

\begin{theorem}
Let $1 \leq r \leq {s \choose d}$. For $s<2d-r+2$, we have 
$$d_{r}(C^{\mathbb{P}}_{d})=(q-2)^{s-d-1}(q-1)^{d-r}[(q-1)^{r}-1].$$
\end{theorem}
\begin{proof}
For $f_{1},~f_{2},\cdots, f_{r} \in KV_{d}$ linearly independent polynomials, we have 
\begin{align*}
|V_{\mathbb{T}}(f_{1},~f_{2},\cdots,f_{r})|&=|V_{\mathbb{T}}(f_{1}^{*},~f_{2}^{*},\cdots,f_{r}^{*})|.\\
\end{align*}
Now, $f_{1}^{*},~f_{2}^{*},\cdots, f_{r}^{*} \in KV_{s-d}$. Also, if $v=s-d$ then $2v+r-2<s$. Thus, by Theorem $4.1$, we have 
$$|V_{\mathbb{T}}(f_{1}^{*},~f_{2}^{*},\cdots,f_{r}^{*})| \leq (q-1)^{s-1}-(q-2)^{v-1}(q-1)^{s-v-r}[(q-1)^{r}-1].$$
Therefore, $$d_{r}(C^{\mathbb{P}}_{d}))\geq (q-2)^{s-d-1}(q-1)^{d-r}[(q-1)^{r}-1].$$

For the converse, consider the polynomials $g'_{1},~g'_{2},\cdots,g'_{r}$ where 
$$g'_{1}:=(t_{1}-t_{2})(t_{3}-t_{4})\cdots (t_{2v-3}-t_{2v-2})(t_{2v-1}-t_{2v})t_{2v+1}t_{2v+2}\cdots t_{s},$$
$$g'_{2}:=(t_{1}-t_{2})(t_{3}-t_{4})\cdots (t_{2v-3}-t_{2v-2})(t_{2v}-t_{2v+1})t_{2v-1}t_{2v+2}t_{2v+3}\cdots t_{s},$$
$$g'_{3}:=(t_{1}-t_{2})(t_{3}-t_{4})\cdots (t_{2v-3}-t_{2v-2})(t_{2v+1}-t_{2v+2})t_{2v-1}t_{2v}t_{2v+3}\cdots t_{s},$$
$$\vdots$$
$$g'_{r}:=(t_{1}-t_{2})(t_{3}-t_{4})\cdots (t_{2v-3}-t_{2v-2})(t_{2v+r-2}-t_{2v+r-1})t_{2v-1}t_{2v} \cdots t_{2v+r-3}t_{2v+r} \cdots t_{s}.$$ \\
Then $g'_{1},~g'_{2},\cdots g'_{r} \in KV_{d}$ and are linearly independent. Let $g:=(t_{1}-t_{2})(t_{3}-t_{4})\cdots (t_{2v-3}-t_{2v-2})$ and $h_{i}:=(t_{2v+i-2}-t_{2v+i-1})$, for $1 \leq i \leq r$. Let $T_{1}=(\mathbb{F}_{q}^{*})^{2v-2}$. Then proceeding as in Theorem $4.1$, we get
\begin{align*}
|V_{T}(g'_{1},g'_{2}, \cdots,g'_{r})|&=|V_{T}(g'_{1})\cap V_{T}(g'_{2})\cap \cdots \cap V_{T}(g'_{r}))|\\
&=|V_{T}(g) \cup (V_{T}(h_{1})\cap V_{T}( h_{2})\cap \cdots \cap V_{T}( h_{r})) |\\
&=(q-1)^{s-2v+2} |V_{T_{1}}(g)| +(q-1)^{s-r}-(q-1)^{s-2v-r+2} |V_{T_{1}}(g)|\\
&=(q-1)^{s-r}+(q-1)^{s-2v-r+2}[(q-1)^{r}-1] \left[ \sum_{i=1}^{v-1} (-1)^{i-1} { v-1 \choose i} (q-1)^{2v-2-i}\right]\\
&=(q-1)^{s}-(q-1)^{s-v-r+1}(q-2)^{v-1}[(q-1)^{r}-1]\\
&=(q-1)^{s}-(q-1)^{d-r+1}(q-2)^{s-d-1}[(q-1)^{r}-1]\\
&=(q-1)| V_{\mathbb{T}}(g'_{1},g'_{2}, \cdots,g'_{r})|
\end{align*}
This implies $$d_{r}(C^{\mathbb{P}}_{d})\leq (q-2)^{s-d-1}(q-1)^{d-r}[(q-1)^{r}-1].$$
This proves the result.
\end{proof}

\begin{corollary}
Let $1 \leq r \leq {s \choose d}$. For $s<2d-r+2$, we have 
$$d_{r}(C_{d})=(q-2)^{s-d-1}(q-1)^{d-r+1}[(q-1)^{r}-1].$$
\end{corollary}

If $r=1$, Theorem $4.1$, Theorem $4.4$, Corollary $4.2$ and Corollary $4.5$ determine the minimum distance of $C_{d}$ and $C^{\mathbb{P}}_{d}$ as in Theorem $2.2$. If $r=2$, these results determine the second generalized Hamming weight of $C_{d}$ and $C^{\mathbb{P}}_{d}$ for $s>2d$ and $s<2d$. Similarly, if $r=3$, we get the third generalized Hamming weight of these codes for $s>2d+1$ and $s<2d-1$. We give estimates on the second and third generalized Hamming weights of $C_{d}$ and $C^{\mathbb{P}}_{d}$, in the rest of the cases.\\

Note that, for $q=2$ or $s=d$, $dim_{\mathbb{F}_{q}} C_{d}=dim_{\mathbb{F}_{q}} C^{\mathbb{P}}_{d}=1$. Thus, the second generalized Hamming weight of these codes doesn't make sense. So, we assume that $q \geq 3$ and $d<s$ when calculating the second generalized weight of these codes. Similarly, we assume that ${s \choose d} \geq 3$ and $q \geq 3$ when calculating the third generalized weight of these codes.\\

\begin{proposition}
For $s=2d$, there exists $f_{1},f_{2} \in KV_{d}$ linearly independent over $K$ such that 
 $$|V_{\mathbb{T}}(f_{1}) \cap V_{\mathbb{T}}(f_{2})|=(q-1)^{s-1}-(q-2)^{d-1}(q-1)^{s-d}.$$ Therefore, $$d_{2}(C^{\mathbb{P}}_{d}) \leq (q-2)^{d-1}(q-1)^{s-d}$$ and $$d_{2}(C_{d}) \leq (q-2)^{d-1}(q-1)^{s-d+1}.$$
\end{proposition}
\begin{proof}
Consider the polynomials $$f''_{1}:=(t_{1}-t_{2})(t_{3}-t_{4})\cdots(t_{2d-3}-t_{2d-2})(t_{2d-1}-t_{2d})$$ and $$f''_{2}:= (t_{1}-t_{2})(t_{3}-t_{4})\cdots(t_{2d-3}-t_{2d-2})t_{2d}.$$ Then $f''_{1},f''_{2} \in KV_{d}$ are linearly independent over $K$. For $1 \leq i \leq d$, let $h_{i}:=(t_{2i-1}-t_{2i})$. Then 
\begin{align*}
|V_{T}(f''_{1}) \cap V_{T}(f''_{2})|&=|(V_{T}(h_{1}h_{2}\cdots h_{d-1}) \cup V_{T}(h_{d})) \cap (V_{T}(h_{1}h_{2}\cdots h_{d-1}) \cup V_{T}(t_{2d}))|\\
&=|V_{T}(h_{1}h_{2}\cdots h_{d-1})|\\
&=\sum_{i=1}^{d-1} (-1)^{i-1} {d-1 \choose i} (q-1)^{s-i}\\
&=(q-1)^{s}-(q-1)^{s-d+1}(q-2)^{d-1}\\
&=(q-1)|V_{\mathbb{T}}(f''_{1}) \cap V_{\mathbb{T}}(f''_{2})|.
\end{align*}
Thus, for $s=2d$, $$d_{2}(C^{\mathbb{P}}_{d})\leq (q-2)^{d-1}(q-1)^{s-d}$$ and $$d_{2}(C_{d})\leq (q-2)^{d-1}(q-1)^{s-d+1}.$$
\end{proof}

\begin{proposition}
Let $s=2d+1$ and ${ s \choose d} \geq 3$. There exists $f_{1},f_{2}, f_{3} \in KV_{d}$ linearly independent over $K$ such that
$$|V_{\mathbb{T}}(f_{1}) \cap V_{\mathbb{T}}(f_{2}) \cap V_{\mathbb{T}}(f_{3})|=(q-1)^{s-1}-(q-2)^{d-1}(q-1)^{s-d}.$$
Therefore, $d_{3}(C^{\mathbb{P}}_{d}) \leq (q-2)^{d-1}(q-1)^{s-d}$ and $d_{3}(C_{d})\leq (q-2)^{d-1}(q-1)^{s-d+1}.$
\end{proposition}
\begin{proof}
Consider the following polynomials
$$f_{1}:=(t_{1}-t_{2})(t_{3}-t_{4}) \cdots (t_{2d-3}-t_{2d-2})t_{2d-1},$$
$$f_{2}:=(t_{1}-t_{2})(t_{3}-t_{4}) \cdots (t_{2d-3}-t_{2d-2})t_{2d},$$
and
$$f_{3}:=(t_{1}-t_{2})(t_{3}-t_{4}) \cdots (t_{2d-3}-t_{2d-2})t_{2d+1}.$$
Let $g:=(t_{1}-t_{2})(t_{3}-t_{4}) \cdots (t_{2d-3}-t_{2d-2})$. Then,
\begin{align*}
\mid V_{T}(f_{1}) \cap V_{T}(f_{2}) \cap V_{T}(f_{3}) \mid &=\mid V_{T}(g) \mid\\
&=\sum_{i=1}^{d-1} (-1)^{i-1} {d-1 \choose i} (q-1)^{s-i}\\
&=(q-1)^{s}-(q-2)^{d-1}(q-1)^{s-d+1}\\
&=(q-1) |V_{\mathbb{T}}(f_{1}) \cap V_{\mathbb{T}}(f_{2}) \cap V_{\mathbb{T}}(f_{3})|.
\end{align*}

This implies $d_{3}(C^{\mathbb{P}}_{d}) \leq (q-2)^{d-1}(q-1)^{s-d}$ and $d_{3}(C_{d})\leq (q-2)^{d-1}(q-1)^{s-d+1}.$\\
\end{proof}

\begin{proposition}
Let $s=2d$ and ${ s \choose d} \geq 3$. There exists $f_{1},f_{2}, f_{3} \in KV_{d}$ linearly independent over $K$ such that
$$|V_{\mathbb{T}}(f_{1}) \cap V_{\mathbb{T}}(f_{2}) \cap V_{\mathbb{T}}(f_{3})|= (q-1)^{s-1}-(q-2)^{d-1}(q-1)^{s-d-2}[q(q-1)-1].$$
Therefore, $$d_{3}(C^{\mathbb{P}}_{d}) \leq (q-2)^{d-1}(q-1)^{d-2}[q(q-1)-1]$$ and $$d_{3}(C_{d}) \leq (q-2)^{d-1}(q-1)^{d-1}[q(q-1)-1].$$
\end{proposition}
\begin{proof}
Let $T_{1}:=(\mathbb{F}_{q}^{*})^{2d-4}$ and consider the polynomials
$$f_{1}:=(t_{1}-t_{2})(t_{3}-t_{4}) \cdots (t_{2d-5}-t_{2d-4})(t_{2d-3}-t_{2d-2})(t_{2d-1}-t_{2d}),$$
$$f_{2}:=(t_{1}-t_{2})(t_{3}-t_{4}) \cdots (t_{2d-5}-t_{2d-4})(t_{2d-3}-t_{2d-1})(t_{2d-2}-t_{2d}),$$
and
$$f_{3}:=(t_{1}-t_{2})(t_{3}-t_{4}) \cdots (t_{2d-5}-t_{2d-4})(t_{2d-3}-t_{2d-2})t_{2d}.$$
~\\
Let $g:=(t_{1}-t_{2})(t_{3}-t_{4}) \cdots (t_{2d-5}-t_{2d-4})$. Then,
\begin{align*}
&|V_{T}(f_{1}) \cap V_{T}(f_{2}) \cap V_{T}(f_{3})|\\ &=|V_{T}(g) \cup V_{T}(t_{2d-3}-t_{2d-1},t_{2d-3}-t_{2d-2}) \cup V_{T}(t_{2d-2}-t_{2d},t_{2d-3}-t_{2d-2})|\\
&=2(q-1)^{s-2}-(q-1)^{s-3}+(q-1)^{s-2d+1}[(q-1)^{3}-2(q-1)+1] |V_{T_{1}}(g)|\\
&=2(q-1)^{s-2}-(q-1)^{s-3}+[(q-1)^{s-2d+4}-2(q-1)^{s-2d+2}\\
&\hspace{3.5 cm}+(q-1)^{s-2d+1}]\left[-\sum_{i=1}^{d-2} (-1)^{i} { d-2 \choose i}(q-1)^{2d-4-i} \right]\\
&=(q-1)^{s}-(q-2)^{d-1}(q-1)^{d-1}[q(q-1)-1]\\
&=(q-1)|V_{\mathbb{T}}(f_{1}) \cap V_{\mathbb{T}}(f_{2}) \cap V_{\mathbb{T}}(f_{3})|.
\end{align*}

Therefore, $$d_{3}(C^{\mathbb{P}}_{d}) \leq (q-2)^{d-1}(q-1)^{d-2}[q(q-1)-1]$$ and $$d_{3}(C_{d}) \leq (q-2)^{d-1}(q-1)^{d-1}[q(q-1)-1].$$
\end{proof}

\begin{proposition}
Let $s=2d-1$ and ${s \choose d} \geq 3$. There exists $f_{1},f_{2}, f_{3} \in KV_{d}$ linearly independent over $K$ such that 
$$\mid V_{\mathbb{T}}(f_{1}) \cap V_{\mathbb{T}}(f_{2}) \cap V_{\mathbb{T}}(f_{3}) \mid = (q-1)^{s-1}-(q-2)^{d-2}(q-1)^{s-d+1}.$$ Therefore, $$d_{3}(C^{\mathbb{P}}_{d}) \leq (q-2)^{d-2}(q-1)^{s-d+1}$$ and $$d_{3}(C_{d}) \leq (q-2)^{d-2}(q-1)^{s-d+2}.$$
\end{proposition}
\begin{proof}
Consider the following polynomials
$$f_{1}:=(t_{1}-t_{2})(t_{3}-t_{4}) \cdots (t_{2d-5}-t_{2d-4})t_{2d-3}t_{2d-2},$$
$$f_{2}:=(t_{1}-t_{2})(t_{3}-t_{4}) \cdots (t_{2d-5}-t_{2d-4})t_{2d-3}t_{2d-1},$$
and
$$f_{3}:=(t_{1}-t_{2})(t_{3}-t_{4}) \cdots (t_{2d-5}-t_{2d-4})t_{2d-2}t_{2d-1}.$$
Let $g:=(t_{1}-t_{2})(t_{3}-t_{4}) \cdots (t_{2d-5}-t_{2d-4})$. Then,
\begin{align*}
|V_{T}(f_{1}) \cap V_{T}(f_{2}) \cap V_{T}(f_{3})| &=|V_{T}(g)|\\
&=\sum_{i=1}^{d-2} (-1)^{i-1} {d-2 \choose i} (q-1)^{s-i}\\
&=(q-1)^{s}-(q-2)^{d-2}(q-1)^{s-d+2}\\
&=(q-1) |V_{\mathbb{T}}(f_{1}) \cap V_{\mathbb{T}}(f_{2}) \cap V_{\mathbb{T}}(f_{3})|.
\end{align*}
Therefore, $d_{3}(C^{\mathbb{P}}_{d}) \leq (q-2)^{d-2}(q-1)^{s-d+1}$ and $d_{3}(C_{d}) \leq (q-2)^{d-2}(q-1)^{s-d+2}.$\\
\end{proof}

\subsection{Generalized Hamming weights of square-free affine evaluation code}
For $1 \leq r \leq dim_{K} KV_{\leq d}$, the $r$-th generalized Hamming weight of $C_{\leq d}$ is given by
$$d_{r}(C_{\leq d}):=min\{~|T \backslash V_{T}(H)|~:~ H:=\{f_{1},~f_{2},\cdots,f_{r}\} \subseteq KV_{\leq d} \text{ is linearly independent over } K\}.$$ In this subsection, we determine the generalized Hamming weights of $C_{\leq d}$, partially.\\

\begin{theorem}
Let $1 \leq r \leq dim_{K} KV_{\leq d}$. For $d+r-2<s$, we have $$d_{r}(C_{\leq d})=(q-2)^{d-1}(q-1)^{s-d-r+1}[(q-1)^{r}-1].$$
\end{theorem}
\begin{proof}
From Theorem $3.9$, we have
$$d_{r}(C_{\leq d})\geq (q-2)^{d-1}(q-1)^{s-d-r+1}[(q-1)^{r}-1].$$
To prove the converse, consider the following polynomials\\
$$f'_{1}=(t_{1}-1)(t_{2}-1)\cdots (t_{d-1}-1)(t_{d}-1),$$
$$f'_{2}=(t_{1}-1)(t_{2}-1)\cdots (t_{d-1}-1)(t_{d+1}-1),$$
$$\vdots$$
$$f'_{r}=(t_{1}-1)(t_{2}-1)\cdots (t_{d-1}-1)(t_{d+r-1}-1).$$

Then, $f'_{1},~f'_{1},\cdots,f'_{r} \in KV_{\leq d}$ are linearly independent over $K$. Let $g=(t_{1}-1)(t_{2}-1)\cdots (t_{d-1}-1)$. Let $T_{1}=(\mathbb{F}_{q}^{*})^{d-1}.$ Then,
\begin{align*}
|V_{T}(f'_{1},~f'_{2},\cdots,f'_{r})|&=|V_{T}(g) \cup V_{T}(t_{d}-1,~t_{d+1}-1,\cdots,t_{d+r-1}-1)|\\
&=(q-1)^{s-d+1}|V_{T_{1}}(g)|+(q-1)^{s-r}-(q-1)^{s-d-r+1}|V_{T_{1}}(g)|\\
&=(q-1)^{s-r}+(q-1)^{s-d-r+1}[(q-1)^{r}-1][(q-1)^{d-1}-(q-2)^{d-1}]\\
&=(q-1)^{s}-(q-2)^{d-1}(q-1)^{s-d-r+1}[(q-1)^{r}-1].\\
\end{align*}
Thus, $d_{r}(C_{\leq d})\leq (q-2)^{d-1}(q-1)^{s-d-r+1}[(q-1)^{r}-1]$. Hence, the result.
\end{proof}

When $r=2$, we get $d_{2}(C_{\leq d})=q(q-2)^{d}(q-1)^{s-d-1}$ for $d<s$, which gives us result of Theorem $2.5$ for $d<s$.

\section{Dual code}
In this section, we determine the dual code of the toric codes over hypersimplices, using the ideas of \cite{J} and \cite{J1}.\\

Consider the set \begin{equation}
\Delta:=\{(a_{1},~a_{2},\cdots,a_{s})~|~a_{i} \in \{0,1\},~\sum_{i=1}^{s}a_{i}=d\}.
\end{equation}
For $(b_{1},~b_{2},\cdots,b_{s})\in (\{0,~1,\cdots,q-2\})^{s}$, define $(\widehat{b_{1}},~\widehat{b_{2}},\cdots,\widehat{b_{s}})\in (\{0,~1,\cdots,q-2\})^{s}$ as 
\[\widehat{b_{i}}
 =\begin{cases}
0 & \text{if } b_{i}=0,\\
q-1-b_{i}  &\text{if } b_{i} \neq 0.\\
\end{cases}  
 \]

 Let $\mathcal{G}:=\{\widehat{b}~:~b \in \Delta\}$ and $\Delta':=(\{0,~1,\cdots,q-2\})^{s} \backslash \mathcal{G}.$ Define $$E_{\Delta'}:=span_{\mathbb{F}_{q}}\{t_{1}^{a_{1}} t_{2}^{a_{2}} \cdots t_{s}^{a_{s}}~:~(a_{1},~a_{2},\cdots,a_{s}) \in \Delta'\}.$$
 
Replacing $KV_{d}$ by $S$ in equation $(2.1)$, we define a map 
\begin{align*}
ev_{T}:~& S \rightarrow \mathbb{F}_{q}^{m}\\
&f \mapsto (f(P_{1}),~f(P_{2}),\cdots, f(P_{m})),
\end{align*}
 where $T=\{ P_{1},P_{2}, \cdots, P_{m}\}$ is the affine torus in $\mathbb{A}^{s}$. Then, $C_{\Delta'}:=ev_{T}(E_{\Delta'})$ is a linear code over $\mathbb{F}_{q}$ of length $m$ on the affine torus $T$. Similarly, we have the code $C_{\Delta}=ev_{T}(E_{\Delta})=C_{d}$. From the definition, it is clear that the map $ev_{T}|_{E_{\Delta'}}$ is injective. Therefore, we have
\begin{lemma}
$dim_{\mathbb{F}_{q}}( C_{\Delta'})=(q-1)^s-{ s \choose d}=m-dim_{\mathbb{F}_{q}}( C_{d}).$
\end{lemma}

\begin{lemma}
For $a \in \Delta$ and $b \in \Delta'$, we have $$ev_{T}(t^{a}).ev_{T}(t^{b})=0,$$ where we have $t^{a}:=t_{1}^{a_{1}} t_{2}^{a_{2}} \cdots t_{s}^{a_{s}}$ for $a=(a_{1},~a_{2},\cdots, a_{s})$.
\end{lemma}
\begin{proof}
Fix a primitive element $\theta$ of $\mathbb{F}_{q}$ i.e. $\mathbb{F}_{q}^{*}=\langle \theta \rangle$. For $a,b \in (\{0,~1,\cdots,q-2\})^{s}$, we have 
$$ev_{T}(t^{a}).ev_{T}(t^{b})=\prod_{j=1}^{s} \sum_{i=0}^{q-2} (\theta^{i})^{a_{j}+b_{j}} .$$
Now, if for some $j$, $a_{j}=b_{j}=0$ or $a_{j}=q-1-b_{j}$ , then $$\sum_{i=0}^{q-2} (\theta^{i})^{a_{j}+b_{j}}=(q-1)\neq 0.$$
But if for some $j$, $a_{j}+b_{j} \not \equiv 0~(\text{mod } q-1)$, then $$\sum_{i=0}^{q-2} (\theta^{i})^{a_{j}+b_{j}}=\frac{({{\theta}^{a_{j}+b_{j}}})^{q-1}-1}{\theta^{a_{j}+b_{j}}-1}=0.$$
Thus,
$ev_{T}(t^{a}).ev_{T}(t^{b}) \neq 0$ if and only if for each $j$, $1\leq j \leq s$, $a_{j}=b_{j}=0$ or $a_{j}=q-1-b_{j}$, i.e.\ $a \in \Delta$ and $b \not \in \Delta'$. Hence proved.
\end{proof}

\begin{theorem}
The dual of the code $C_{d}$ is the code $C_{\Delta'}$ with respect to Euclidean scalar product i.e.\ $C_{\Delta'}=C_{d}^{\perp}.$
\end{theorem}
\begin{proof}
Let $f \in E_{\Delta'}$. Then $f$ can be written as $f=\sum_{b \in \Delta'} \alpha_{b}t^{b}$ where $\alpha_{b} \in \mathbb{F}_{q}$. For any $g \in E_{\Delta}$, $g=\sum_{a \in \Delta} \beta_{a}t^{a}$, $\beta_{a} \in \mathbb{F}_{q}$, we have 
\begin{align*}
ev_{T}(g).ev_{T}(f)&=ev_{T} \left(\sum_{a \in \Delta} \beta_{a}t^{a} \right).ev_{T} \left(\sum_{b \in \Delta'} \alpha_{b}t^{b}\right)\\
&=\sum_{a \in \Delta} \sum_{b \in \Delta'} \alpha_{b} \beta_{a} ev_{T}(t^{a}).ev_{T}(t^{b})=0,
\end{align*}
by Lemma $5.2$. This implies $C_{\Delta'} \subseteq C_{d}^{\perp}.$ From Lemma $5.1$, we have $dim_{\mathbb{F}_{q}}(C_{\Delta'})=dim_{\mathbb{F}_{q}}(C_{d}^{\perp})$. Hence the result.
\end{proof}

The codes $C_{d}$ and $C_{\Delta'}$ are \textit{J}-affine variety codes with \textit{J}$=\{1,~2,\cdots,s\}$, as studied in \cite{J},\cite{J1},\cite{J2}, etc. By using Corollary 2 from \cite{J}, we can obtain stabilizer codes.

\subsection{Dual of $C^{\mathbb{P}}_{d}$}
We have $\mathbb{T}=\{1\} \times (\mathbb{F}_{q}^{*})^{s-1}$. With $\Delta$ as in equation $(5.1)$, let $$\mathcal{H}_{1}:=\{(a_{2},~a_{3},\cdots,a_{s})~:~(a_{1},~a_{2},\cdots,a_{s}) \in \Delta\}.$$ For $(c_{2},~c_{3},\cdots,c_{s}) \in (\{0,~1,\cdots,q-2\})^{s-1}$, define $(\widehat{c_{2}},~\widehat{c_{2}},\cdots,\widehat{c_{s}})\in (\{0,~1,\cdots,q-2\})^{s-1}$ as \[\widehat{c_{i}}
 =\begin{cases}
0 & \text{if } c_{i}=0,\\
q-1-c_{i}  &\text{if } c_{i} \neq 0.\\
\end{cases}  
 \]
Let $\mathcal{H}_{2}:=\{(\widehat{b_{2}},~\widehat{b_{3}},\cdots,\widehat{b_{s}})~:~(b_{2},~b_{3},\cdots,b_{s}) \in \mathcal{H}_{1}\}.$ Let $\mathcal{U}:=(\{0,~1,\cdots,q-2\})^{s-1} \backslash \mathcal{H}_{2}$.\\\\
Define $E_{\mathcal{U}}:=span_{\mathbb{F}_{q}}\{t_{2}^{a_{2}} t_{3}^{a_{3}}\cdots t_{s}^{a_{s}}~|~(a_{2},~a_{3},\cdots,a_{s}) \in \mathcal{U}\}$.\\

Let $T':=(\mathbb{F}_{q}^{*})^{s-1}$. Then $|T'|=\bar{m}$. Let $T'=\{R_{1},~R_{2},\cdots,R_{\bar{m}}\}$ such that $Q_{i}=(1, R_{i})$, $1 \leq i \leq \bar{m}$, where $Q_{i} \in \{1 \} \times (\mathbb{F}_{q}^{*})^{s-1}$, $1 \leq i \leq \bar{m}$, as defined in section $2$. Define a map
\begin{align*}
ev_{T'}:~& S \rightarrow \mathbb{F}_{q}^{\bar{m}}\\
&f \mapsto (f(R_{1}),~f(R_{2}),\cdots, f(R_{\bar{m}})).
\end{align*}
Define $C_{\mathcal{U}}:=ev_{T'}(E_{\mathcal{U}})$. Then $C_{\mathcal{U}}$ is a linear code over $\mathbb{F}_{q}$ of length $\bar{m}$.\\

Note that 
\begin{align*}
C^{\mathbb{P}}_{d}&=\{(f(Q_{1}),\cdots, f(Q_{\bar{m}})):~ f \in KV_{d}\}\\
&=\{(f(1,R_{1}),\cdots, f(1,R_{\bar{m}})):~ f \in KV_{d}\}\\
&=\{(g(R_{1}),\cdots, g(R_{\bar{m}})):~ g \in E_{\mathcal{H}_{1}}\}=ev_{T'}(E_{\mathcal{H}_{1}}),\\
\end{align*}
where $E_{\mathcal{H}_{1}}$ is the $\mathbb{F}_{q}$-vector space generated by the set $\{t^{a}:=t_{2}^{a_{2}} t_{3}^{a_{3}}\cdots t_{s}^{a_{s}} ~:~ (a_{2},~a_{3}, \cdots, a_{s}) \in \mathcal{H}_{1}\}$. Then, $dim_{\mathbb{F}_{q}} C_{\mathcal{U}}=(q-1)^{s-1}-{s \choose d}=\bar{m}-dim_{\mathbb{F}_{q}}(C^{\mathbb{P}}_{d})$.\\

Following Lemma $5.2$, we get
\begin{lemma}
For $a \in \mathcal{H}_{1}$ and $b \in \mathcal{U}$, we have $$ev_{T'}(t^{a}).ev_{T'}(t^{b})=0.$$ 
\end{lemma}
We have the final result.
\begin{theorem}
The dual code of $C^{\mathbb{P}}_{d}$ is the code $C_{\mathcal{U}}$ with respect to Euclidean scalar product i.e $$(C^{\mathbb{P}}_{d})^{\perp}=C_{\mathcal{U}}.$$
\end{theorem}

\section{Concluding remarks}
In this note, we have determined the generalized Hamming weights of toric codes over hypersimplices. The generalized Hamming weights of square-free affine evaluation codes are also calculated, under certain conditions. Furthermore, we have determined the dual of the toric codes with respect to the Euclidean scalar product. It will be interesting to calculate the remaining generalized Hamming weights of these codes.

\section{Acknowledgements}
We thank  Delio Jaramillo, Maria Vaz Pinto and Rafael H. Villarreal for suggesting this problem. We also thank Maria Vaz Pinto for suggesting improvements in Lemma $3.3$. \par
This work is a part of the PhD thesis of the first author.

\end{document}